\documentclass[12pt]{article}
\usepackage[top=2.5cm,bottom=2.5cm,left=2.9cm,right=2.9cm]{geometry}
\usepackage{amssymb}
\usepackage{amsmath,amsthm}
\usepackage[latin1]{inputenc}
\usepackage{graphicx}
\usepackage{hyperref}
\usepackage{enumerate}
\usepackage{tikz}
\usepackage{tkz-graph}
\usepackage{mathrsfs}
\usepackage{verbatim}
\usepackage{upgreek}
 \usepackage{mathptmx}

\usepackage{upgreek}
\usepackage{mathptmx}
\usepackage{eucal}

\usepackage{colortbl}

\hypersetup{colorlinks=true}

\hypersetup{colorlinks=true, linkcolor=blue, citecolor=blue,urlcolor=blue}

\newtheorem{remark}{Remark}[section]

\newtheorem{lemma}[remark]{Lemma}
\newtheorem{theorem}[remark]{Theorem}
\newtheorem{proposition}[remark]{Proposition}

\newtheorem{corollary}[remark]{Corollary}

\title{Double domination in lexicographic product graphs}

\author{A. Cabrera Mart\'inez$^{(1)}$, S. Cabrera Garc\'ia$^{(2)}$, J. A. Rodr\'{\i}guez-Vel\'{a}zquez$^{(1)}$\\
$^{(1)}${\small Universitat Rovira i Virgili }\\{\small Departament d'Enginyeria Inform\`atica i Matem\`atiques } \\  {\small Av. Pa\"{\i}sos
Catalans 26, 43007 Tarragona, Spain.} \\{\small
  abel.cabrera\@@urv.cat, juanalberto.rodriguez\@@urv.cat}\\
$^{(2)}${\small Universitat Polit\'ecnica de Valencia}\\{\small Departamento de Estad\'istica e Investigaci\'on Operativa Aplicadas y Calidad}\\{\small Camino de Vera s/n, 46022 Valencia, Spain.}\\ {\small suicabga\@@eio.upv.es}}

\date{ }
\begin{document}
\maketitle

\begin{abstract}
In a graph $G$, a vertex dominates itself and its neighbours. A subset $S\subseteq V(G)$ is said to be a double dominating set of $G$ if $S$ dominates every vertex of $G$ at least twice. The minimum cardinality among all double dominating sets of $G$ is the double domination number. In this article, we  obtain tight bounds  and closed formulas for the double domination number of lexicographic product graphs $G\circ H$ in terms 
 of invariants of the factor graphs $G$ and $H$.
\end{abstract}

{\it Keywords}:
Double domination; total domination; total Roman $\{2\}$-domination;  lexicographic product

\section{Introduction}

In a graph $G$, a vertex dominates itself and its neighbours.
A subset $S\subseteq V(G)$ is said to be a \emph{dominating set} of $G$ if $S$ dominates every vertex of $G$, while $S$ is said to be a \emph{double dominating set} of $G$ if $S$ dominates every vertex of $G$ at least twice. 
A subset $S\subseteq V(G)$ is said to be a \emph{total dominating set} of $G$ if every vertex $v\in V(G)$ is dominated by at least one vertex in $S\setminus \{v\}$. 
  The minimum cardinality among all  dominating sets of $G$ is the   \emph{domination number}, denoted by $\gamma(G)$. The \emph{double domination number}  and the \emph{total domination number} of $G$ are defined by analogy, and are denoted by $\gamma_{\times 2}(G)$ and $\gamma_t(G)$, respectively. The domination number and the total domination number have been extensively studied. For instance, we cite the following books  \cite{book1,book2,book-total-dom}. 
 The double domination number, which has been less studied, was introduced in \cite{haynes2000} by Harary and Haynes, and studied further in a number of works including \cite{chellali2006, chellali2005, Majid2019,  harant2005, Khelifi2012}.

Let $f: V(G) \rightarrow \{0,1,2\}$ be a function. For any $i\in \{0,1,2\}$ we define the subsets of vertices $V_i=\{v\in V(G): f(v)=i\}$ and we  identify $f$ with the three subsets of $V(G)$ induced by $f$. Thus, in order to emphasize the notation of these sets, we denote the function by $f(V_0, V_1, V_2)$. Given a set $X\subseteq V(G)$, we define $f(X)=\sum_{v\in X}f(v)$, and the \emph{weight} of $f$ is defined to be $\omega(f)=f(V(G))=|V_1|+2|V_2|$.

A function $f(V_0,V_1,V_2)$ is a \emph{total Roman dominating function} (TRDF) on a graph $G$ if $V_1\cup V_2$ is a total dominating set and $N(v)\cap V_2\ne \emptyset$ for every vertex $v\in V_0$, where $N(v)$ denotes the \emph{open neighbourhood} of $v$. 
This concept was introduced by Liu and Chang \cite{Liu2013}. For recent results on total Roman domination in graphs we cite \cite{AbdollahzadehAhangarHenningSamodivkinEtAl2016, TRDF2019, Cabrera2020, Dorota2018}.

A  function $f(V_0,V_1,V_2)$  is a \emph{total Roman $\{2\}$-dominating function} (TR2DF) if $V_1\cup V_2$ is a total dominating set and  $f(N(v))\geq 2$ for every vertex $v\in V_0$.  This concept was recently introduced in \cite{TR2DF2019}. Notice that $S\subseteq V(G)$ is a double dominating set of $G$ if and only if there exists a TR2DF   $f(V_0,V_1,V_2)$ such that $V_1=S$ and $V_2=\emptyset.$ 

The \emph{total Roman domination number}, denoted by $\gamma_{tR}(G)$, is the minimum weight among all TRDFs on $G$. By analogy, we define the  \emph{total Roman $\{2\}$-domination number}, which is denoted by $\gamma_{t\{R2\}}(G)$.

Notice that, by definition, $\gamma_{\times 2}(G)\ge \gamma_{t\{R2\}}(G)$. As an example of graph $G$  for which $\gamma_{\times 2}(G)> \gamma_{t\{R2\}}(G)$ we consider a star graph $K_{1,r}$ for $r\ge 3$. In this case,  $\gamma_{\times 2}(K_{1,r})=r+1>3= \gamma_{t\{R2\}}(K_{1,r})$.
We would point out that the problem of characterizing all graphs with $\gamma_{\times 2}(G)=\gamma_{t\{R2\}}(G)$ remains open. In this paper we show that the values of these two parameters coincide for any lexicographic product graph $G\circ H$ in which graph $G$ has no isolated vertices and graph $H$ is not trivial. Furthermore, we obtain tight bounds and  closed formulas  for $\gamma_{\times 2}(G\circ H)$ in terms of invariants of the factor graphs $G$ and $H$.

\subsection{Additional concepts, notation and tools}

All graphs considered in this paper are finite and undirected, without loops or multiple edges. As usual, the \emph{closed neighbourhood} of a vertex $v\in V(G)$ is denoted by $N[v]=N(v)\cup \{v\}$. We say that a vertex $v\in V(G)$ is a \emph{universal vertex} of $G$ if $N[v]=V(G)$. By analogy with the notation used for vertices, for a set $S \subseteq V(G)$, its \emph{open neighbourhood} is the set $N(S)=\cup_{v\in S} N(v)$, and its \emph{closed neighbourhood} is the set $N[S]= N(S)\cup S$. The subgraph induced by $S\subseteq V(G)$ will be denoted by  $\langle S\rangle$, while the graph obtained from $G$ by removing all the vertices in $S\subseteq V(G)$ (and all the edges incident with a vertex in $S$) will be denoted by $G-S$.

We will use the notation $K_n$, $K_{1,n-1}$, $C_n$, $N_n$,  $P_n$  and $K_{n,n-r}$ for complete graphs, star graphs, cycle graphs, empty graphs, path graphs and complete bipartite graphs of order $n$, respectively. A double star $S_{n_1,n_2}$ is the graph obtained by joining the center of two stars $K_{1,n_1}$ and $K_{1,n_2}$ with an edge.

Given two graphs $G$ and $H$, the \emph{lexicographic product} of $G$ and $H$ is the graph $G\circ H$ whose vertex set is $V(G\circ  H) = V(G)\times V(H)$ and $(u,v)(x,y)\in E(G \circ H)$ if and only if $ux \in E(G)$ or $u=x$ and $vy \in E(H)$. Notice that for any vertex $u\in V(G)$ the subgraph of $G\circ H$ induced by $\{u\}\times V(H)$ is isomorphic to $H$. For simplicity, we will denote this subgraph by $H_u$. For basic properties of lexicographic product  graphs we suggest the books \cite{Hammack2011, Imrich2000}. 
In particular,   we cite the following works on domination theory of lexicographic product graphs:   standard domination   \cite{MR3363260,Nowakowski1996,Zhang2011},   Roman domination  \cite{SUmenjak:2012:RDL:2263360.2264103},  total Roman domination \cite{Dorota2018}, weak Roman domination   \cite{Valveny2017},   rainbow domination   \cite{MR3057019},  $k$-rainbow independent domination \cite{MR3897459},
 super domination    \cite{Dettlaff-LemanskaRodrZuazua2017},  twin domination  \cite{MR3644816}, power domination \cite{MR2399365} and doubly connected domination   \cite{MR3200151}.
 
 For simplicity, for  any $(u,v) \in V(G)\times V(H)$ and any TR2DF $f$ on $G \circ H$ we write $N(u,v)$ and $f(u,v)$ instead of $N((u,v))$ and $f((u,v))$, respectively.

For the remainder of the paper, definitions will be introduced whenever a concept is needed.

Now we present some tools that will be very useful throughout the work. 

\begin{proposition}\label{prop-inequalities}{\rm\cite{ TR2DF2019}}
The following inequalities hold for any graph $G$ with no isolated vertex.
\begin{enumerate}[{\rm(i)}]
  \item $\gamma_t(G)\leq \gamma_{t\{R2\}}(G)\leq \gamma_{tR}(G)\leq 2\gamma_t(G)$.
  \item $ \gamma_{t\{R2\}}(G)\leq \gamma_{\times 2}(G)$.
\end{enumerate}
\end{proposition}

A double dominating set of cardinality  $\gamma_{\times 2}(G)$ will be called a $\gamma_{\times 2}(G)$-set. A similar agreement will be assumed when referring to optimal sets (and functions) associated to other parameters used in the article.

\begin{theorem}\label{teo-x2=t}
If $\gamma_{\times 2}(G)=\gamma_t(G)$, then for any  $\gamma_{\times 2}(G)$-set $D$  there exists an  integer $k\geq 1$ such that $\langle D\rangle\cong \cup_{i=1}^kK_2$.
\end{theorem}

\begin{proof}
Let $D$ be a $\gamma_{\times 2}(G)$-set and suppose that $\langle D\rangle$ has a component $G'$ which is not isomorphic to  $K_2$. Let $v\in V(G')$ be a vertex of minimum degree in $G'$. Notice that the set $D\setminus \{v\}$ is a total dominating set of $G$. Hence, $\gamma_t(G)\leq |D\setminus \{v\}|<|D|=\gamma_{\times 2}(G)$, which is a contradiction. Therefore, the result follows. 
\end{proof}

\begin{theorem}\label{teo-tR2-2t}{\rm\cite{TR2DF2019}}
The following statements are equivalent.
   \begin{itemize}
     \item $\gamma_{t\{R2\}}(G)=2\gamma_t(G)$.  
     \item  $\gamma_{t\{R2\}}(G)=\gamma_{tR}(G)$ and $\gamma_t(G)=\gamma(G)$.
   \end{itemize}
\end{theorem}

The following theorem merges two results obtained in \cite{TR2DF2019}  and  \cite{haynes2000}.

\begin{theorem}[{\rm \cite{TR2DF2019}} and {\rm\cite{haynes2000}}]\label{teo-x2=2}
The following statements are equivalent.
   \begin{itemize}
     \item $\gamma_{t\{R2\}}(G)=2$. 
     \item $\gamma_{\times 2}(G)=2$.
     \item $G$ has at least two universal vertices.
   \end{itemize}
\end{theorem}

It is readily seen that if $G'$ is a spanning subgraph of $G$, then any $\gamma_{\times 2}(G')$-set is a double dominating set of $G$. Therefore, the following result is immediate.  

\begin{theorem}\label{DoubleDominSpanningSubgraph}
If  $G'$ is a spanning subgraph of $G$ with no isolated vertex, then
$$\gamma_{\times 2}(G)\le \gamma_{\times 2}(G').$$
\end{theorem} 

In Proposition \ref{FormulaGammax2LexicographicPnxH=CnxH} we will show some cases of lexicographic product graphs for which the equality above holds. 

\begin{remark}\label{rem-Cn-Pn}
For any integer $n\geq 3$,

\begin{enumerate}[{\rm (i)}]
\item $\gamma_{t\{R2\}}(P_n)\stackrel{\mbox{\rm \cite{TR2DF2019}}}{=}\gamma_{\times 2}(P_n)\stackrel{\mbox{\rm \cite{chellali2006}}}{=}\left\{\begin{array}{ll}
                                 2\left\lceil \frac{n}{3}\right\rceil + 1, & \mbox{if $n\equiv 0 \pmod 3$,} \\[5pt]
                                 2\left\lceil \frac{n}{3}\right\rceil, & \mbox{otherwise.}
                               \end{array}
\right.$

\item $\gamma_{t\{R2\}}(C_n)\stackrel{\mbox{\rm \cite{TR2DF2019}}}{=}\gamma_{\times 2}(C_n)\stackrel{\mbox{\rm \cite{haynes2000}}}{=}                                 \left\lceil \frac{2n}{3}\right\rceil.$
\end{enumerate}
\end{remark}

The next theorem merges two results obtained in \cite{SUmenjak:2012:RDL:2263360.2264103} and \cite{Zhang2011}.

\begin{theorem}[{\rm\cite{SUmenjak:2012:RDL:2263360.2264103}} and {\rm\cite{Zhang2011}}]\label{teo-char-gamma}
For any graph $G$ with no isolated vertex and any nontrivial graph $H$, 
$$\gamma(G\circ H)=\left\{\begin{array}{ll}
                                 \gamma(G), & \mbox{if $\gamma(H)=1$,} \\[5pt]
                                 \gamma_t(G), & \mbox{if $\gamma(H)\geq 2$.}
                               \end{array}
\right.$$
\end{theorem}

\begin{theorem}\label{teo-char-gamma-t}{\rm\cite{TWRDF(LEX)2019}}
For any graph $G$ with no isolated vertex and any nontrivial graph $H$, 
$$\gamma_t(G\circ H)=\gamma_t(G).$$
\end{theorem}

\section{Main results on lexicographic product graphs}

Our first result shows that the double domination number  and the total Roman $\{2\}$-domination number coincide for  lexicographic product graphs.

\begin{theorem}\label{teo-equality-TR2D-DD}
For any graph $G$ with no isolated vertex and any nontrivial graph $H$,
$$\gamma_{\times 2}(G\circ H)=\gamma_{t\{R2\}}(G\circ H).$$
\end{theorem}

\begin{proof}
Proposition \ref{prop-inequalities} (ii) leads to $\gamma_{\times 2}(G\circ H)\ge \gamma_{t\{R2\}}(G\circ H)$.  Let $f(V_0,V_1,V_2)$ be a $\gamma_{t\{R2\}}(G\circ H)$-function such that $|V_2|$ is minimum. Suppose that $\gamma_{\times 2}(G\circ H)>\gamma_{t\{R2\}}(G\circ H)$.  In such a case, $V_2\neq \emptyset$ and we can differentiate two cases for a fixed vertex $(u,v)\in V_2$. 

\vspace{.2cm}
\noindent
Case 1. $N(u,v)\cap (V_1\cup V_2)\subseteq V(H_u)$. In this case, for any $(u',v')\in N(u)\times V(H)$ we define the function $g(V_0',V_1',V_2')$ where $V_0'=V_0\setminus\{(u',v')\}$, $V_1'=V_1\cup \{(u,v),(u',v')\}$ and $V_2'=V_2\setminus \{(u,v)\}$. Observe that $V_1'\cup V_2'$ is a total dominating set of $G\circ H$ and every vertex $w\in V_0'\subseteq V_0$ satisfies that $g(N(w))\geq 2$. Hence, $g$ is a $\gamma_{t\{R2\}}(G\circ H)$-function and $|V_2'|=|V_2|-1$, which is a contradiction.

\vspace{.2cm}
\noindent   
Case 2. $N(u)\times V(H)\cap (V_1\cup V_2)\neq \emptyset$.   If $f(u,v')>0$ for every vertex $v'\in V(H)$, then the function $g$, defined by $g(u,v)=1$ and $g(x,y)=f(x,y)$ whenever $(x,y)\in V(G\circ H)\setminus \{(u,v)\}$, is a TR2DF on $G\circ H$ and $\omega(g)=\omega(f)-1$, which is a contradiction. Hence, there exists a vertex $v'\in V(H)$ such that $f(u,v')=0$. In this case, we define the function $g(V_0',V_1',V_2')$ where $V_0'=V_0\setminus\{(u,v')\}$, $V_1'=V_1\cup \{(u,v),(u,v')\}$ and $V_2'=V_2\setminus \{(u,v)\}$. Notice that $V_1'\cup V_2'$ is a total dominating set of $G\circ H$ and every vertex $w\in V_0'\subseteq V_0$ satisfies that $g(N(w))\geq 2$. Hence, $g$ is a $\gamma_{t\{R2\}}(G\circ H)$-function and $|V_2'|=|V_2|-1$, which is a contradiction again.  

According to the two cases above, we deduce that $V_2=\emptyset$. Therefore, $V_1$ is a $\gamma_{\times 2}(G\circ H)$-set and so $\gamma_{\times 2}(G\circ H)=\gamma_{t\{R2\}}(G\circ H).$ 
\end{proof}

From now on, the main goal is to obtain tight bounds or closed formulas for $\gamma_{\times 2}(G\circ H)$ and express them in terms of invariants of $G$ and $H$. 

A set $X\subseteq V(G)$ is called a $2$-\emph{packing} if $N[u]\cap N[v] = \emptyset$ for every pair of different vertices $u,v \in X$, \cite{book2}. The $2$-\emph{packing number} $\rho(G)$ is the maximum cardinality among all $2$-packing sets of $G$. As usual, a $2$-packing of cardinality $\rho(G)$ is called a $\rho(G)$-set.

\begin{theorem}\label{teo-bounds-x2}  
For any graph $G$ with no isolated vertex and any nontrivial graph $H$, 
$$\max\{\gamma_t(G),2\rho(G)\}\leq \gamma_{\times 2}(G\circ H)\leq 2\gamma_t(G).$$
\end{theorem}

\begin{proof}
By Proposition \ref{prop-inequalities} (i) and Theorem \ref{teo-char-gamma-t} we deduce that 
$$\gamma_t(G)=\gamma_t(G\circ H)\leq \gamma_{\times 2}(G\circ H)\leq 2\gamma_t(G\circ H)=2\gamma_t(G).$$ 
Now, for any $\rho(G)$-set $X$ and any $\gamma_{\times 2}(G\circ H)$-set $D$ we have that
$$\gamma_{\times 2}(G\circ H)=|D|=\sum_{u\in V(G)}|D\cap V(H_u)|\geq \sum_{u\in X}\sum_{w\in N[u]}|D\cap V(H_w)|\geq 2|X|=2\rho(G).$$
Therefore, the proof is complete.
\end{proof}

We would point out that the upper bound $\gamma_{\times 2}(G\circ H)\leq \min\{2\gamma_t(G),\gamma(G)\gamma_{\times 2}(H)\}$ was proposed in \cite{Cuivillas2014} for the particular case in which $G$ and $H$ are connected. Obviously, the connectivity is not needed, and the bound $\gamma_{\times 2}(G\circ H)\leq \gamma(G)\gamma_{\times 2}(H)$ also holds for  any graph $G$ (even if $G$ is empty) and any graph $H$ with no isolated vertices. 

In Theorem \ref{CharactGammax2Lexic=2gammat} we will show cases in which $\gamma_{\times 2}(G\circ H)=2\gamma_t(G)$, while in Theorem \ref{teo-G-gamma=1-consequences} (i) and (ii) we will show cases in which $\gamma_{\times 2}(G\circ H)=2\rho(G)$ or $\gamma_{\times 2}(G\circ H)=\gamma_t(G)$.

\begin{corollary}\label{Coroteo-bounds-st}  
If $\gamma(G)=1$, then for any nontrivial graph $H$, 
$$2\leq \gamma_{\times 2}(G\circ H)\leq 4.$$
\end{corollary}
In Section \ref{SectionCharacterizationSmallValues} we characterize the graphs with $\gamma_{\times 2}(G\circ H)\in \{2,3\}$. Hence, by Corollary~\ref{Coroteo-bounds-st} the graphs with $\gamma_{\times 2}(G\circ H)=4$ will be automatically characterized whenever $\gamma(G)=1$.

\begin{theorem}\label{CharactGammax2Lexic=2gammat}
If $G$ is a graph with no isolated vertex and $H$ is a nontrivial graph, then the following statements are equivalent.

\begin{enumerate}[{\rm (a)}]
\item $\gamma_{\times 2}(G\circ H)=2\gamma_t(G)$.
\item $\gamma_{\times 2}(G\circ H)=\gamma_{tR}(G\circ H)$ and $(\gamma_t(G)=\gamma(G)$ or $\gamma(H)\geq 2)$.
\end{enumerate} 
\end{theorem}

\begin{proof}
Assume that   $\gamma_{\times 2}(G\circ H)=2\gamma_t(G)$. By Theorems \ref{teo-char-gamma-t} and \ref{teo-equality-TR2D-DD} we deduce that 
$$
  \gamma_{t\{R2\}}(G\circ H)=\gamma_{\times 2}(G\circ H)=2\gamma_t(G)=2\gamma_{t}(G\circ H).
$$
Hence,  by Theorem \ref{teo-tR2-2t} we have that $\gamma_{\times 2}(G\circ H)=\gamma_{tR}(G\circ H)$ and $\gamma(G\circ H)=\gamma_t(G\circ H)=\gamma_t(G)$. Notice that $\gamma_t(G\circ H)=\gamma_t(G)$ if and only if  $\gamma_t(G)=\gamma(G)$ or $\gamma(H)\geq 2$, by Theorem  \ref{teo-char-gamma}. Therefore, (b) follows.

Conversely, assume that (b)  holds. 
By Theorem  \ref{teo-equality-TR2D-DD} we have that 
\begin{equation}\label{EqTh2.1}
  \gamma_{t\{R2\}}(G\circ H)=\gamma_{\times 2}(G\circ H)=\gamma_{tR}(G\circ H).
\end{equation}
 Now, if  $\gamma_t(G)=\gamma(G)$ or $\gamma(H)\geq 2$, by Theorems \ref{teo-char-gamma} and  \ref{teo-char-gamma-t} we deduce that
\begin{equation}\label{Eqxx}
\gamma_{t}(G\circ H)=\gamma_t(G)=\gamma(G\circ H).
\end{equation}
Hence, Theorem \ref{teo-tR2-2t} and equations  \eqref{EqTh2.1} and \eqref{Eqxx} lead to 
$\gamma_{\times 2}(G\circ H)=\gamma_{t\{R2\}}(G\circ H)=2\gamma_{t}(G\circ H)=2\gamma_{t}(G)$, as required. 
\end{proof}

It was shown in \cite{Cockayne1980} that for any connected graph $G$ of order $n\geq 3$, $\gamma_t(G)\leq \frac{2n}{3}$. Hence, Proposition \ref{prop-inequalities} (i) and Theorem \ref {teo-equality-TR2D-DD} lead to the following result. 

\begin{theorem}\label{prop-n}
For any connected graph $G$ of order $n\geq 3$ and any graph $H$,
$$\gamma_{\times 2}(G\circ H)\leq 2\left\lfloor\frac{2n}{3}\right\rfloor.$$
\end{theorem}

In order to show that the bound above is tight, we consider the case of rooted product graphs. Given a graph $G$  and a graph $H$ with root $v\in V(H)$, the rooted product $G\bullet_v H$ is defined as the graph obtained from $G$ and $H$ by taking one copy of $G$ and $|V(G)|$ copies of $H$ and identifying the  $i^{th}$ vertex of $G$ with  vertex $v$ in the $i^{th}$ copy of $H$ for every $i\in \{1,\dots,|V(G)|\}$. For instance, the graph $P_5\bullet_v P_3$ where $v$ is a leaf, is shown in Figure \ref{fig-P5-P2}. Later, when we read Lemma \ref{lemma-conditions-proj}, it will be easy to see that for $n=|V(G\bullet_v P_3)|=3|V(G)|$ we have that $\gamma_{\times 2}((G\bullet_v P_3)\circ H)=4|V(G)|=2\left\lfloor\frac{2n}{3}\right\rfloor$ whenever $\gamma(H)\geq 3$.

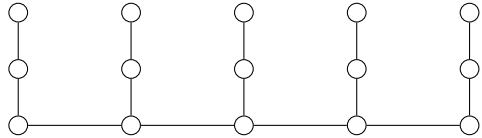
\begin{figure}[ht]
\centering
\begin{tikzpicture}[scale=1, transform shape]

\node [draw, shape=circle,inner sep=2.5pt ] (a2) at  (0,1.5) {};
\node [draw, shape=circle, inner sep=2.5pt] (a3) at  (0,0.75) {};
\node [draw, shape=circle, inner sep=2.5pt] (a4) at  (0,0) {};

\node [draw, shape=circle, inner sep=2.5pt] (b2) at  (1.5,1.5) {};
\node [draw,shape=circle,inner sep=2.5pt] (b3) at  (1.5,0.75) {};
\node [draw, shape=circle,inner sep=2.5pt] (b4) at  (1.5,0) {};

\node [draw, shape=circle,inner sep=2.5pt] (c2) at  (3,1.5) {};
\node [draw,shape=circle,inner sep=2.5pt] (c3) at  (3,0.75) {};
\node [draw, shape=circle,inner sep=2.5pt] (c4) at  (3,0) {};

\node [draw, shape=circle,inner sep=2.5pt] (d2) at  (4.5,1.5) {};
\node [draw,shape=circle,inner sep=2.5pt] (d3) at  (4.5,0.75) {};
\node [draw, shape=circle,inner sep=2.5pt] (d4) at  (4.5,0) {};

\node [draw, shape=circle,inner sep=2.5pt] (e2) at  (6,1.5) {};
\node [draw, shape=circle,inner sep=2.5pt] (e3) at  (6,0.75) {};
\node [draw, shape=circle,inner sep=2.5pt] (e4) at  (6,0) {};

\draw(a4)--(b4)--(c4)--(d4)--(e4);
\draw(a2)--(a3)--(a4);
\draw(b2)--(b3)--(b4);
\draw(c2)--(c3)--(c4);
\draw(d2)--(d3)--(d4);
\draw(e2)--(e3)--(e4);

\end{tikzpicture}
\caption{The graph $P_5\bullet_v P_3$}\label{fig-P5-P2}
\end{figure}

\begin{lemma}\label{obs-vertice<=2}
For any  graph $G$ with no isolated vertex and any nontrivial graph $H$, there exists a $\gamma_{\times 2}(G\circ H)$-set $S$ such that $|S\cap V(H_u)|\leq 2$, for every $u\in V(G)$.
\end{lemma}

\begin{proof}
Given a double dominating set $S$ of  $G\circ H$, we define the set $S_3=\{x\in V(G): \, |S\cap V(H_x)|\ge 3\}$. Let $S$ be a $\gamma_{\times 2}(G\circ H)$-set such that $|S_3|$ is minimum among all $\gamma_{\times 2}(G\circ H)$-sets. If $|S_3|=0$, then we are done. Hence, we suppose that there exists $u\in S_3$ and let $(u,v)\in S$.  We assume that $|S\cap V(H_u)|$ is minimum among all vertices in $S_3$.
It is readily seen that if there exists $u'\in N(u)$ such that $|S\cap V(H_{u'})|\ge 2$, then $S'=S\setminus \{(u,v)\}$ is a double dominating set of $G\circ H$, which is a contradiction. Hence, if $u'\in N(u)$, then $|S\cap V(H_{u'})|\le  1$, and in this case it is not difficult to check that for $(u',v')\notin  S$ the set $S''=(S\setminus \{(u,v)\})\cup \{(u',v')\}$ is a  $\gamma_{\times 2}(G\circ H)$-set such that $|S''_3|$ is minimum among all $\gamma_{\times 2}(G\circ H)$-sets. If $|S''_3|<|S_3|$, then we obtain a contradiction, otherwise $u\in S''_3$  and  $|S''\cap V(H_u)|$ is minimum among all vertices in $S''_3$, so that we can successively repeat this process, until obtaining a contradiction.
 Therefore, the result follows.  
\end{proof}

\begin{theorem}\label{teo-bounds-conditions-H}
Let $G$ be a graph with no isolated vertex and let $H$ be a nontrivial graph.
\begin{enumerate}[{\rm (i)}]
\item If $\gamma(H)=1$, then $\gamma_{\times 2}(G\circ H)\leq \gamma_{t\{R2\}}(G)$.
\item If $H$ has at least two universal vertices, then $\gamma_{\times 2}(G\circ H)\leq 2\gamma(G)$.
\item If $H$ has exactly one universal vertex, then $\gamma_{\times 2}(G\circ H)=\gamma_{t\{R2\}}(G)$.
\item If $\gamma(H)\geq 2$, then $\gamma_{\times 2}(G\circ H)\geq \gamma_{t\{R2\}}(G)$.
\end{enumerate}
\end{theorem}

\begin{proof}
Let $f$ be a $\gamma_{t\{R2\}}(G)$-function and let $v$ be a universal vertex of $H$. Let $f'$ be the function defined by $f'(u,v)=f(u)$ for every $u\in V(G)$ and $f'(x,y)=0$ whenever $x\in V(G)$ and $y\in V(H)\setminus \{v\}$. It is readily seen  that $f'$ is a TR2DF on $G\circ H$. Hence, by Theorem \ref{teo-equality-TR2D-DD} we conclude that $\gamma_{\times 2}(G\circ H)=\gamma_{t\{R2\}}(G\circ H)\leq \omega(f')=\omega(f)=\gamma_{t\{R2\}}(G)$ and (i) follows. 

Let $D$ be a $\gamma(G)$-set and let $y_1,y_2$ be two universal vertices of $H$. It is not difficult to see that $S=D\times \{y_1,y_2\}$  is a double dominating set of $G\circ H$. Therefore, $\gamma_{\times 2}(G\circ H)\leq |S|=2\gamma(G)$ and (ii) follows.

From now on, let $S$ be a $\gamma_{\times 2}(G\circ H)$-set that satisfies Lemma \ref{obs-vertice<=2} and assume that either $\gamma(H)\geq 2$ or $H$ has exactly one universal vertex.  Let $g(V_0,V_1,V_2)$ be the function defined by $g(u)=|S\cap V(H_u)|$  for every $u\in V(G)$. We claim that $g$ is a TR2DF on $G$. It is clear  that every vertex in $V_1$ has to be adjacent to some vertex in  $V_1\cup V_2$ and, if $\gamma(H)\geq 2$ or $H$ has exactly one universal vertex, then  by Theorem \ref{teo-x2=2} we have that $\gamma_{\times 2}(H)\geq 3$, which implies that every vertex in $V_2$ has to be adjacent to some vertex in  $V_1\cup V_2$. Hence,  $V_1\cup V_2$ is a total dominating set of $G$. Now, if $x\in V_0$, then $S\cap V(H_x)=\emptyset$, and so $|N(V(H_x))\cap S|\geq 2$. Thus, $g(N(x))\geq 2$, which implies that $g$ is TR2DF on $G$ and so $\gamma_{t\{R2\}}(G)\leq \omega(g)=|S|=\gamma_{\times 2}(G\circ H)$. Therefore, (iii) and (iv) follow.
\end{proof}

The following result is a direct consequence of Theorems \ref{teo-bounds-x2} and \ref{teo-bounds-conditions-H}. Recall that $\gamma_{\times 2}(H)=2$ if and only if $H$ has at least two universal vertices (see Theorem \ref{teo-x2=2}). 

\begin{theorem}\label{teo-G-gamma=1-consequences}
Let $G$ be  a graph with no isolated vertex and let $H$ be a  nontrivial graph. 
\begin{enumerate}[{\rm (i)}]
\item If $\gamma(G)=\rho(G)$ and $\gamma_{\times 2}(H)=2$, then $\gamma_{\times 2}(G\circ H)=2\gamma(G).$
\item   If  $\gamma_{t\{R2\}}(G)\in \{\gamma_t(G),2\rho(G)\}$ and $\gamma(H)=1$, then $\gamma_{\times 2}(G\circ H)=\gamma_{t\{R2\}}(G).$
 \item  If $\gamma_{t\{R2\}}(G)=2\gamma_t(G)$ and $\gamma(H)\geq 2$, then $\gamma_{\times 2}(G\circ H)=\gamma_{t\{R2\}}(G)$.
 \end{enumerate}
\end{theorem}

It is well known that $\gamma(T)=\rho(T)$ for any tree $T$. Hence, the following corollary is a direct consequence of Theorem \ref{teo-G-gamma=1-consequences}.

\begin{corollary}\label{cor-closed-formulae-st=2}
For any tree $T$ and any graph $H$ with $\gamma_{\times 2}(H)=2$, $$\gamma_{\times 2}(T\circ H)= 2\gamma(T).$$
\end{corollary}

A \emph{double total dominating set} of a graph $G$ is a set $S$ of vertices of $G$ such that every vertex in $V(G)$ is adjacent to at least two vertices in $S$  \cite{book-total-dom}. The \emph{double total domination number} of $G$, denoted by $\gamma_{2,t}(G)$, is the minimum cardinality among all double total dominating sets.

\begin{theorem}{\rm \cite{Valveny2017}}\label{DoubletotalLexicographico}
If $G$ is a graph of minimum degree greater than or equal to two, then for any graph $H$,
$$\gamma_{2,t}(G\circ H)\leq \gamma_{2,t}(G).$$
\end{theorem} 

\begin{theorem}\label{teo-bounds-x2-2,t}
Let $G$ be a graph of minimum degree greater than or equal to two and order $n$. The following statements hold.

\begin{enumerate}[{\rm (i)}]
\item For any graph $H$, $\gamma_{\times 2}(G\circ H)\leq \gamma_{2,t}(G).$ 
\item For any graph $H$,  $\gamma_{\times 2}(G\circ H)\leq n.$
\end{enumerate}
\end{theorem}

\begin{proof}
Since every double total dominating set is a double dominating set, we deduce that $\gamma_{\times 2}(G\circ H)\leq \gamma_{2,t}(G\circ H)$. Hence, from Theorem \ref{DoubletotalLexicographico} we deduce (i).  Finally, since $\gamma_{2,t}(G)\le n$, from (i) we deduce (ii).
\end{proof}

The following family $\mathcal{H}_k$ of graphs was shown in \cite{Valveny2017}. A graph $G$ belongs to $\mathcal{H}_k$ if and only if  it is constructed from a cycle $C_k$ and $k$ empty graphs $N_{s_1},\dots, N_{s_k}$ of order $s_1,\dots,s_k$, respectively, and joining by an edge each vertex from $N_{s_i}$ with the  vertices $v_i$ and $v_{i+1}$ of $C_k$. Here we are assuming that $v_i$ is adjacent to $v_{i+1}$ in $C_k$, where the subscripts are taken modulo $k$. Figure \ref{The graphs in H}  shows a graph $G$ belonging to $\mathcal{H}_k$, where $k=4$, $s_1=s_3=3$ and  $s_2=s_4=2$.

Notice that $\gamma_{t\{R2\}}(G)=\gamma_{2,t}(G)$, for every $G\in \mathcal{H}_k$. Hence, from Theorems \ref{teo-bounds-conditions-H} (iv) and \ref{teo-bounds-x2-2,t} (i) we deduce that $\gamma_{\times 2}(G\circ H)= \gamma_{2,t}(G)$ for any $G\in \mathcal{H}_k$ and any graph $H$ such that $\gamma(H)\geq 2$.
  
\begin{figure}[ht]
\centering
\begin{tikzpicture}[scale=1, transform shape]
\node [draw, fill=black, shape=circle, inner sep=2.5pt] (c1) at  (0,0) {};
\node [draw, fill=black, shape=circle,inner sep=2.5pt ] (c2) at  (1.5,0) {};
\node [draw, fill=black, shape=circle, inner sep=2.5pt] (c3) at  (1.5,1.5) {};
\node [draw, fill=black, shape=circle, inner sep=2.5pt] (c4) at  (0,1.5) {};
\node [draw, shape=circle, inner sep=2.5pt] (b1) at  (2.5,0.75) {};
\node [draw, shape=circle,inner sep=2.5pt ] (b2) at  (3.5,0.75) {};
\node [draw, shape=circle, inner sep=2.5pt] (b3) at (4.5,0.75) {};
\node [draw, shape=circle, inner sep=2.5pt] (bb1) at (-1,0.75) {};
\node [draw, shape=circle, inner sep=2.5pt] (bb2) at  (-2,0.75) {};
\node [draw, shape=circle,inner sep=2.5pt ] (bb3) at  (-3,0.75) {};
\node [draw, shape=circle, inner sep=2.5pt] (a1) at (0.75,2.5) {};
\node [draw, shape=circle,inner sep=2.5pt ] (a2) at  (0.75,3.5) {};
\node [draw, shape=circle, inner sep=2.5pt] (aa1) at  (0.75,-1) {};
\node [draw, shape=circle, inner sep=2.5pt] (aa2) at  (0.75,-2) {};

\draw(c1)--(c2)--(c3)--(c4)--(c1);

\draw(c1)--(bb1)--(c4)--(bb2)--(c1)--(bb3)--(c4);

\draw(c2)--(b1)--(c3)--(b2)--(c2)--(b3)--(c3);

\draw(c1)--(aa1)--(c2)--(aa2)--(c1);

\draw(c4)--(a1)--(c3)--(a2)--(c4);
\end{tikzpicture}
\caption{The set of black-coloured vertices is a $\gamma_{2,t}(G)$-set.}\label{The graphs in H}
\end{figure}
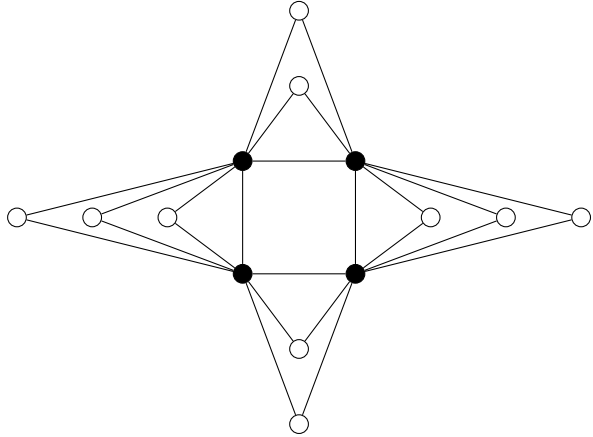


\section{Small values of $\gamma_{\times 2}(G\circ H)$}\label{SectionCharacterizationSmallValues}
First, we characterize the graphs with $\gamma_{\times 2}(G\circ H)=2$.

\begin{theorem}\label{teo-char-st=2}
For any nontrivial graph $G$ and any graph $H$, the following statements are equivalent.

\begin{enumerate}[{\rm (i)}]
\item $\gamma_{\times 2}(G\circ H)=2$. 
\item $\gamma(G)=\gamma(H)=1$ and $(\gamma_{\times 2}(G)=2$ or $\gamma_{\times 2}(H)=2)$.
\end{enumerate}  
\end{theorem}

\begin{proof}
Notice that $G\circ H$ has at least two universal vertices if and only if  $\gamma(G)=\gamma(H)=1$, and  also $G$ has at least two universal vertices  or $H$ has at least two universal vertices. Hence,  by Theorem \ref{teo-x2=2} we conclude that (i) and (ii) are equivalent.
\end{proof}

Next, we characterize the graphs that satisfying $\gamma_{\times 2}(G\circ H)=3$. Before we shall need the following definitions.
For a set $S\subseteq V(G\circ H)$ we define the following subsets of $V(G)$.
$$\mathcal{A}_S= \{v\in V(G):\, |S\cap V(H_v)|\ge 2\};$$
$$\mathcal{B}_S= \{v\in V(G):\, |S\cap V(H_v)|=1\};$$
$$\mathcal{C}_S= \{v\in V(G):\, S\cap V(H_v)=\emptyset\}.$$

\begin{theorem}\label{teo-char-st=3}
For any nontrivial graphs $G$ and $H$, $\gamma_{\times 2}(G\circ H)=3$ if and only if one of the following conditions is satisfied.

\begin{enumerate}[{\rm (i)}]
\item $G\cong P_2$ and $\gamma(H)=2$.
\item $G\not\cong P_2$ has at least two universal vertices and $\gamma(H)\geq 2$.
\item $G$ has exactly one universal vertex and either  $\gamma(H)=2$ or $H$ has exactly one universal vertex.
\item $G$ has exactly one universal vertex, $\gamma_{2,t}(G)=3$ and $\gamma(H)\ge 3$.

\item $\gamma(G)=2$ and  $\gamma_{2,t}(G)=3$.
\item $\gamma(G)=2$, $\gamma_{\times 2}(G)=3<\gamma_{2,t}(G)$ and $\gamma(H)=1$.
\end{enumerate}
\end{theorem}

\begin{proof}
Notice that with the above premises, $G$ does not have isolated vertices. Let $S$ be a $\gamma_{\times 2}(G\circ H)$-set that satisfies Lemma \ref{obs-vertice<=2} and assume that $|S|=3$. By Theorems \ref{teo-char-gamma-t} and \ref{teo-x2=t}  we have that $3=\gamma_{\times 2}(G\circ H)>\gamma_t(G\circ H)=\gamma_t(G)\geq 2$, which implies that $\gamma_t(G)=2$ and so $\gamma(G)\in \{1,2\}$. We differentiate two cases.
 
\vspace{0,3cm}
\noindent Case 1. $\gamma(G)=1$. In this case,  Theorem \ref{teo-char-st=2} leads to $\gamma_{\times 2}(H)\geq 3$.  Now, we consider the following subcases.

\vspace{0,2cm}
\noindent 
Subcase 1.1. $G\cong P_2$.  Notice that Theorem \ref{teo-char-st=2} leads to $\gamma(H)\geq 2$. Suppose that $\gamma(H)\geq 3$ and let $V(G)=\{u,w\}$. Observe that $S\cap V(H_u)\ne \emptyset$ and  $S\cap V(H_w)\ne \emptyset$. Without loss of generality, let $S\cap V(H_u)=\{(u,v_1),(u,v_2)\}$ and $|S\cap V(H_w)|=1$. 
 Since $\gamma(H)\geq 3$, we have that $\{v_1,v_2\}$ is not a dominating set of $H$, which implies that no vertex in $\{u\}\times (V(H)\setminus (N(v_1)\cup N(v_2))$ has two neigbours in $S$, which is a contradiction. Hence $\gamma(H)=2$. Therefore, (i) follows. 
 
\vspace{0,2cm}
\noindent 
Subcase 1.2. $G\not\cong P_2$ has at least two universal vertices. In this case, $\gamma_{\times 2}(G)=2 $ and by Theorem \ref{teo-char-st=2} we deduce that $\gamma(H)\geq 2$. Thus,  (ii) follows.  
 
 \vspace{0,2cm}
\noindent 
Subcase 1.3. $G$ has exactly one universal vertex.  If $\gamma(H)\leq 2$, then by Theorem \ref{teo-char-st=2} we deduce that either  $\gamma(H)=2$ or $H$ has exactly one universal vertex, so that (iii) follows.
Assume that $\gamma(H)\geq 3$. Recall that $|S\cap V(H_x)|\le 2$ for every $x\in V(G)$. Now, if there exist two vertices $u,w\in V(G)$ and two vertices $v_1,v_2\in V(H)$ such that $S\cap V(H_u)=\{(u,v_1),(u,v_2)\}$ and $|S\cap V(H_w)|=1$, then we deduce that  no vertex in $\{u\}\times (V(H)\setminus (N(v_1)\cup N(v_2))$ has two neighbours in $S$, which is a contradiction. Therefore, $\mathcal{A}_S=\emptyset$ and $\mathcal{B}_S$ has to be a $\gamma_{2,t}(G)$-set, as every vertex $x\in V(G)$ satisfies 
$|N(x)\cap \mathcal{B}_S|\geq 2$. Therefore, (iv) follows.

\vspace{0,2cm}
\noindent 
Case 2. $\gamma(G)=2$. In this case,  Theorem \ref{teo-x2=2} leads to $\gamma_{\times 2}(G)\geq 3$. 
If there exist two vertices $u,w\in V(G)$ such that $\mathcal{A}_S=\{u\}$ and $\mathcal{B}_S=\{w\}$, then $\{u,w\}$ is a $\gamma_t(G)$-set, and so for any $x\in N(w)\setminus N[u]$ we have that no vertex in $V(H_x)$ has two neighbours in $S$, which is a contradiction. Therefore, $\mathcal{A}_S=\emptyset$ and $|\mathcal{B}_S|=3$, which implies that $\mathcal{B}_S$ is a $\gamma_{\times 2}(G)$-set. Notice that either $\langle\mathcal{B}_S\rangle\cong C_3$ or $\langle\mathcal{B}_S\rangle\cong P_3$. In the first case, $\mathcal{B}_S$ is a $\gamma_{2,t}(G)$-set and (v) follows. Now, assume that $\langle\mathcal{B}_S\rangle\cong P_3$. If $\gamma(H)\ge 2$, then for any vertex $x$ of degree one in $\langle\mathcal{B}_S\rangle$ we have that $V(H_x)$ have vertices  which do  not have two neighbours in $S$, which is a contradiction. Therefore, $\gamma(H)=1$ and if $\gamma_{\times 2}(G)=\gamma_{2,t}(G)$, then $G$ satisfies (v), otherwise $G$ satisfies (vi), by Theorem \ref{teo-bounds-x2-2,t}.

Conversely, notice that if $G$ and $H$ satisfy one of the six conditions above, then Theorem~\ref{teo-char-st=2} leads to $\gamma_{\times 2}(G\circ H)\geq 3$. To conclude that $\gamma_{\times 2}(G\circ H)=3$, we proceed to show how to define  a double dominating set $D$ of  $G\circ H$ of cardinality three for each of the six conditions. 
 \begin{enumerate}[{\rm (i)}]
\item Let $\{v_1,v_2\}$ be a $\gamma(H)$-set and $V(G)=\{u,w\}$. In this case, $D=\{(u,v_1),(u,v_2),(w,v_1)\}$. 
\item Let $u,w\in V(G)$ be two universal vertices, $z\in V(G)\setminus \{u,w\}$ and $v\in V(H)$. In this case, $D=\{(u,v),(w,v),(z,v)\}$.  
\item Let $u$ be a universal vertex of $G$ and $w\in V(G)\setminus \{u\}$. If $\{v_1,v_2\}$ is a $\gamma(H)$-set or $v_1$ is a universal vertex of $H$ and $v_2\in V(H)\setminus \{v_1\}$, then we set $D=\{(u,v_1),(u,v_2),(w,v_1)\}$.
\item Let $X$ be a $\gamma_{2,t}(G)$-set and $v\in V(H)$. In this case, $D=X\times \{v\}$.

\item Let $X$ be a $\gamma_{2,t}(G)$-set and $v\in V(H)$. In this case, $D=X\times \{v\}$.
\item Let $X$ be a $\gamma_{\times 2}(G)$-set and $v$  a universal vertex of $H$. In this case, $D=X\times \{v\}$.
\end{enumerate}

It is readily seen that in all cases $D$ is a  double dominating set of $G\circ H$. Therefore, $\gamma_{\times 2}(G\circ H)=3$.
\end{proof}

The following result, which is a direct consequence of Theorems \ref{teo-bounds-x2}, \ref{teo-char-st=2} and \ref{teo-char-st=3}, shows the cases when $G$ is isomorphic to a complete graph or  a star graph.

\begin{proposition}\label{prop-Kn-K1n}
Let $H$ be a nontrivial graph. For any integer $n\geq 3$, the following statements hold.

\begin{enumerate}[{\rm (i)}]
\item $\displaystyle
\gamma_{\times 2}(K_n\circ H)=\left\{ \begin{array}{ll}
             2 & if \,\, \gamma(H)=1,\\[5pt]
             3 &  otherwise. 
                                  \end{array}\right.
$
\item $
\displaystyle\gamma_{\times 2}(K_{1,n-1}\circ H)=\left\{ \begin{array}{ll}
             2 & if \, \,  \gamma_{\times 2}(H)=2,\\[5pt]
             3 & if \, \,  \gamma_{\times 2}(H)\geq 3 \,\, and \,\, \gamma(H)\leq 2,\\[5pt]
             4 &  otherwise. 
                                  \end{array}\right.
$
\end{enumerate}
\end{proposition}
  
We now consider the cases in which $G$ is a double star graph or a complete bipartite graph. The following result is a direct consequence of Theorems \ref{teo-bounds-x2},  \ref{teo-char-st=2} and \ref{teo-char-st=3}.   
  
\begin{proposition}\label{prop-Snm-Knm}
Let $H$ be a graph. For any integers $n_2\geq n_1\geq 2$, the following statements hold.

\begin{enumerate}[{\rm (i)}]
\item $\gamma_{\times 2}(S_{n_1,n_2}\circ H)=4.$

\item $
\displaystyle\gamma_{\times 2}(K_{n_1,n_2}\circ H)=\left\{ \begin{array}{ll}
             3 & if \, n_1=2 \,\, and \,\, \gamma(H)=1;\\
             4 & otherwise.
                                  \end{array}\right.
$
\end{enumerate}
\end{proposition}

\section{All cases where  $G\cong P_n$ or $G\cong C_n$}\label{subsection-exact-value}

\subsection{Cases where $\gamma(H)=1$}  

\begin{proposition}\label{teo-Pn-gamma=1}
Let $n\geq 3$ be an integer and let $H$ be a nontrivial graph. If 
$\gamma(H)=1$, then
$$
\gamma_{\times 2}(P_n\circ H)=\left\{ \begin{array}{ll}
                     2\left\lceil\frac{n}{3}\right\rceil+1, & \text{if } \gamma_{\times 2}(H)\geq 3 \text{ and } n\equiv 0 ( \text{mod } 3)  ,\\[5pt]
                     2\left\lceil\frac{n}{3} \right\rceil, &  \mbox{otherwise.} 
                                  \end{array}\right.
                                  $$
\end{proposition}

\begin{proof}
If $\gamma_{\times 2}(H)=2$, then by Corollary \ref{cor-closed-formulae-st=2} we deduce that $\gamma_{\times 2}(P_n\circ H)=2\gamma(P_n)$. Now, if $\gamma_{\times 2}(H)\geq 3$, then $H$ has exactly one universal vertex and by Theorem \ref{teo-bounds-conditions-H} (iii) we deduce that $\gamma_{\times 2}(G\circ H)=\gamma_{t\{R2\}}(P_n)$.
\end{proof}

From now on we assume that $V(C_n)=\{u_1,\dots ,u_n\}$, where the subscripts are taken modulo $n$ and consecutive vertices are adjacent.

\begin{proposition}\label{teo-Cn-gamma=1}
Let $n\geq 3$ be an integer and let $H$ be a graph. If 
$\gamma(H)=1$, then
$$\gamma_{\times 2}(C_n\circ H)=\left\lceil  \frac{2n}{3}\right\rceil.$$
\end{proposition}

\begin{proof}
If $H$ is a trivial graph, then we are done, by Remark \ref{rem-Cn-Pn}. From now on we assume that $H$ has at least two vertices.
If $\gamma(H)=1$, then by combining Theorem \ref{teo-bounds-conditions-H} (i)  and Remark \ref{rem-Cn-Pn} (ii), we deduce that $\gamma_{\times 2}(C_n\circ H)\le \left\lceil\frac{2n}{3}\right\rceil$.

Now, let $S$ be a $\gamma_{\times 2}(C_n\circ H)$-set.
Notice that for any $i\in \{1,\dots , n\}$ we have that $$\left|S\cap \left(\bigcup_{j=0}^2 V(H_{u_{i+j}})\right)\right|\ge 2.$$
Hence,
$$3\gamma_{\times 2}(C_n\circ H)=3|S|   =\sum_{i=1}^n  \left|S\cap\left(\bigcup_{j=0}^2 V(H_{u_{i+j}})\right)\right|\ge 2n.$$
Therefore, $\gamma_{\times 2}(C_n\circ H)\ge \left\lceil\frac{2n}{3}\right\rceil$, and the result follows.
\end{proof}

\subsection{Cases where $\gamma(H)=2$} 

To begin this subsection we need to state the following four lemmas.

\begin{lemma}\label{lemma-conditions-proj}
Let $G$ be a nontrivial connected graph and let $H$ be a graph. The following statements hold for every $\gamma_{\times 2}(G\circ H)$-set $S$ that satisfies Lemma \ref{obs-vertice<=2}.

\begin{enumerate}[{\rm (i)}]
\item If $\gamma(H)\geq 2$ and $x\in \mathcal{B}_S\cup \mathcal{C}_S$, then $\displaystyle\sum_{u\in N(x)}|S\cap V(H_{u})|\geq 2$.
\item If $\gamma(H)=2$ and $x\in \mathcal{A}_S$, then $\displaystyle\sum_{u\in N(x)}|S\cap V(H_{u})|\geq 1$.
\item If $\gamma(H)\geq 3$ and $x\in V(G)$, then $\displaystyle\sum_{u\in N(x)}|S\cap V(H_{u})|\geq 2$.
\end{enumerate}  
\end{lemma}

\begin{proof}
First, we suppose that $\gamma(H)=2$. If there exists either a vertex $x\in \mathcal{B}_S\cup \mathcal{C}_S$ such that $\sum_{u\in N(x)}|S\cap V(H_{u})|\leq 1$ or a vertex $x\in \mathcal{A}_S$ such that $\sum_{u\in N(x)}|S\cap V(H_{u})|=0$, then there exists a vertex in $ V(H_x)\setminus S$ which does not have two neighbours in $S$. Therefore,  (ii) follows, and (i) follows for $\gamma(H)=2$. Now, let $x\in V(G)$. Since $S$ satisfies Lemma~\ref{obs-vertice<=2}, if $\gamma(H)\geq 3$, then there exists a vertex in $ V(H_x)\setminus S$ which does not have  neighbours in $S\cap V(H_x)$, which implies that  $\sum_{u\in N(x)}|S\cap V(H_{u})|\geq 2$ and so (i) and (iii) follows. Therefore, the proof is complete.
\end{proof}

 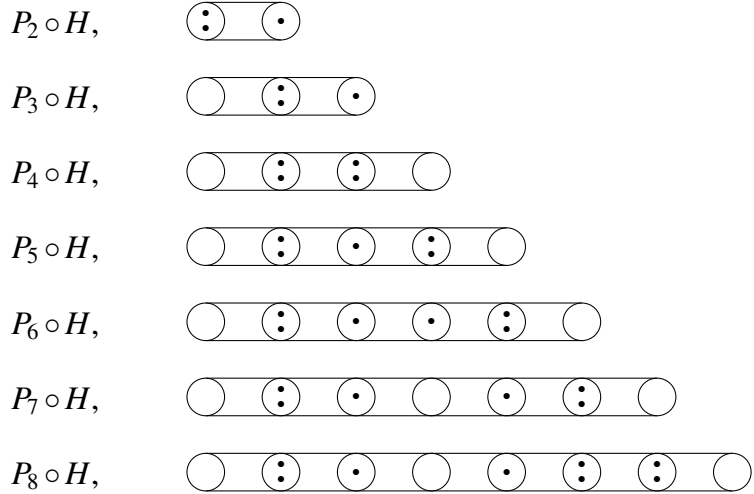
\begin{figure}[ht]
\centering
\begin{tikzpicture}[scale=.5, transform shape]

\node at (-4,1.9) {\huge $P_2\circ H$,};
\draw (0,2) circle (0.5);
\draw (2,2) circle (0.5);

\draw(0,2.5)--(2,2.5);
\draw(0,1.5)--(2,1.5);

\draw (0,2.2) node {$\bullet$};
\draw (0,1.8) node {$\bullet$};
\draw (2,2) node {$\bullet$};
\node at (-4,-0.1) {\Huge $P_3\circ H$,};
\draw (0,0) circle (0.5);
\draw (2,0) circle (0.5);
\draw (4,0) circle (0.5);

\draw(0,0.5)--(4,0.5);
\draw(0,-0.5)--(4,-0.5);

\draw (2,0.2) node {$\bullet$};
\draw (2,-0.2) node {$\bullet$};
\draw (4,0) node {$\bullet$};


\node at (-4,-2.1) {\Huge $P_4\circ H$,};
\draw (0,-2) circle (0.5);
\draw (2,-2) circle (0.5);
\draw (4,-2) circle (0.5);
\draw (6,-2) circle (0.5);

\draw(0,-1.5)--(6,-1.5);
\draw(0,-2.5)--(6,-2.5);

\draw (2,-1.8) node {$\bullet$};
\draw (2,-2.2) node {$\bullet$};

\draw (4,-1.8) node {$\bullet$};
\draw (4,-2.2) node {$\bullet$};


\node at (-4,-4.1) {\Huge $P_5\circ H$,};
\draw (0,-4) circle (0.5);
\draw (2,-4) circle (0.5);
\draw (4,-4) circle (0.5);
\draw (6,-4) circle (0.5);
\draw (8,-4) circle (0.5);

\draw(0,-3.5)--(8,-3.5);
\draw(0,-4.5)--(8,-4.5);

\draw (2,-3.8) node {$\bullet$};
\draw (2,-4.2) node {$\bullet$};

\draw (4,-4) node {$\bullet$};

\draw (6,-3.8) node {$\bullet$};
\draw (6,-4.2) node {$\bullet$};


\node at (-4,-6.1) {\Huge $P_6\circ H$,};
\draw (0,-6) circle (0.5);
\draw (2,-6) circle (0.5);
\draw (4,-6) circle (0.5);
\draw (6,-6) circle (0.5);
\draw (8,-6) circle (0.5);
\draw (10,-6) circle (0.5);

\draw(0,-5.5)--(10,-5.5);
\draw(0,-6.5)--(10,-6.5);

\draw (2,-5.8) node {$\bullet$};
\draw (2,-6.2) node {$\bullet$};

\draw (4,-6) node {$\bullet$};
\draw (6,-6) node {$\bullet$};

\draw (8,-5.8) node {$\bullet$};
\draw (8,-6.2) node {$\bullet$};


\node at (-4,-8.1) {\Huge $P_7\circ H$,};
\draw (0,-8) circle (0.5);
\draw (2,-8) circle (0.5);
\draw (4,-8) circle (0.5);
\draw (6,-8) circle (0.5);
\draw (8,-8) circle (0.5);
\draw (10,-8) circle (0.5);
\draw (12,-8) circle (0.5);

\draw(0,-7.5)--(12,-7.5);
\draw(0,-8.5)--(12,-8.5);

\draw (2,-7.8) node {$\bullet$};
\draw (2,-8.2) node {$\bullet$};

\draw (4,-8) node {$\bullet$};
\draw (8,-8) node {$\bullet$};

\draw (10,-7.8) node {$\bullet$};
\draw (10,-8.2) node {$\bullet$};


\node at (-4,-10.1) {\Huge $P_8\circ H$,};
\draw (0,-10) circle (0.5);
\draw (2,-10) circle (0.5);
\draw (4,-10) circle (0.5);
\draw (6,-10) circle (0.5);
\draw (8,-10) circle (0.5);
\draw (10,-10) circle (0.5);
\draw (12,-10) circle (0.5);
\draw (14,-10) circle (0.5);

\draw(0,-9.5)--(14,-9.5);
\draw(0,-10.5)--(14,-10.5);

\draw (2,-9.8) node {$\bullet$};
\draw (2,-10.2) node {$\bullet$};

\draw (4,-10) node {$\bullet$};
\draw (8,-10) node {$\bullet$};

\draw (10,-9.8) node {$\bullet$};
\draw (10,-10.2) node {$\bullet$};

\draw (12,-9.8) node {$\bullet$};
\draw (12,-10.2) node {$\bullet$};

\end{tikzpicture}
\caption{The scheme used in the proof of Lemma \ref{LemmeGammax2LexicographicPathasxH}.}\label{FigureUpperBoundlexicographicPathsxH}
\end{figure}

\begin{lemma}\label{LemmeGammax2LexicographicPathasxH} For any integer $n\ge 3$ and any graph $H$ with $\gamma(H)=2$,
$$\gamma_{\times 2}(P_n\circ H)\le\left\{ \begin{array}{ll}
n-\left\lfloor \frac{n}{7} \right\rfloor +1 & \text{if} \,\, n\equiv 1,2 \pmod 7,
\\[5pt]
n-\left\lfloor \frac{n}{7} \right\rfloor & otherwise.
\end{array}\right. $$ 
\end{lemma}
 
 \begin{proof}
In Figure \ref{FigureUpperBoundlexicographicPathsxH} we show how to construct a double dominating set $S$ of $P_n\circ H$ for $n\in \{2,\dots, 8\}.$ In this scheme, the  circles represent the copies of $H$ in $P_n\circ H$, two dots in a circle represent two vertices belonging to $S$, which form a dominating set of the corresponding copy of $H$, while a single dot in a circle  represents one vertex belonging to $S$. 

We now proceed to describe the construction of $S$ for any $n=7q+r$, where $q\ge 1$ and $0\le r \le 6$. We partition $V(P_n)=\{u_1, \dots, u_n\}$ into $q$ sets of cardinality $7$ and for $r\ge 1$  one additional set of cardinality $r$, in such a way that the subgraph induced by all these sets are paths.
For any $r\ne 1$, the restriction of $S$ to each of these $q$ paths of length $7$ corresponds to the scheme associated with $P_7\circ H$ in Figure \ref{FigureUpperBoundlexicographicPathsxH}, while for the path of length $r$ (if any) we take the scheme associated with $P_r\circ H$. The case $r=1$ and $q\ge 2$ is slightly different, as for the first $q-1$ paths of length $7$ we take the scheme associated with $P_7\circ H$ and for the path associated with the last  $8$ vertices of $P_n$ we take the scheme associated     with $P_8\circ H$.

Notice that, for $n\equiv 1,2\pmod 7$, we have that $\gamma_{\times 2}(P_n\circ H)\le |S|=6q+r+1=n-\left\lfloor \frac{n}{7} \right\rfloor +1$, while for $n\not \equiv 1,2\pmod 7$ we have   $\gamma_{\times 2}(P_n\circ H)\le |S|=6q+r=n-\left\lfloor \frac{n}{7} \right\rfloor$. Therefore, the result follows.
 \end{proof}

\begin{lemma}\label{lema-1}
Let $P_7=w_1,\ldots ,w_7$ be  a subgraph of $C_n$. Let $H$ be a graph such that $\gamma(H)=2$ and $W=\{w_1,\ldots ,w_7\}\times V(H)$. If  $S$ is a double dominating set  of $C_n\circ H$ which satisfies Lemma \ref{obs-vertice<=2}, then $$|S\cap W|\geq 6.$$
\end{lemma}

\begin{proof}
By Lemma \ref{lemma-conditions-proj} (i) and (ii) we have that $|S\cap (\{w_1,w_2,w_3\}\times V(H))|\geq 2$ and $|S\cap (\{w_4,w_5,w_6,w_7\}\times V(H))|\geq 3$. If $|S\cap (\{w_1,w_2,w_3\}\times V(H))|\geq 3$, then we are done. Hence, we assume that $|S\cap (\{w_1,w_2,w_3\}\times V(H))|=2$. In this case, and by applying again Lemma \ref{lemma-conditions-proj} (i) and (ii) we deduce that $|S\cap (\{w_4,w_5,w_6,w_7\}\times V(H))|\geq 4$, which implies that $|S\cap W|\geq 6$, as desired. Therefore, the proof is complete.  
\end{proof}

\begin{lemma}\label{LemmaDDS-CiclesxHGmmaH=2}
For any integer $n\geq 3$ and any graph $H$ with $\gamma(H)=2$,

$$
\gamma_{\times 2}(C_n\circ H)\geq\left\{ \begin{array}{ll}
             n-\lfloor\frac{n}{7}\rfloor+1 & \text{if } \,  n\equiv 1,2 \pmod 7,\\[5pt]
             n-\lfloor\frac{n}{7}\rfloor & \mbox{otherwise.}
                                  \end{array}\right.
$$
\end{lemma}

\begin{proof}
It is easy to check that $\gamma_{\times 2}(C_n\circ H)=n$ for every $n\in\{3,4,5,6\}$. Now, let $n=7q+r$, with $0\leq r\leq 6$ and $q\geq 1$. Let $S$ be a $\gamma_{\times 2}(C_n\circ H)$-set that satisfies Lemma \ref{obs-vertice<=2}. 

If $r=0$, then by Lemma \ref{lema-1} we have that $|S|\geq 6q=n-\lfloor\frac{n}{7}\rfloor.$  From now on we assume that $r\ge 1$. By Theorem \ref{DoubleDominSpanningSubgraph} and Lemma \ref{LemmeGammax2LexicographicPathasxH} we deduce that$\gamma_{\times 2}(C_n\circ H)\le \gamma_{\times 2}(P_n\circ H)<n$, which implies  that  $\mathcal{A}_S\ne  \emptyset$, otherwise there  exists $u\in V(C_n)$ such that  $N(u)\cap \mathcal{C}_S\ne \emptyset$ and so $|N(u)\cap \mathcal{B}_S|\le 1$, which is a contradiction. Let $x\in \mathcal{A}_S$ and, without loss of generality, we can label the vertices of $C_n$ in such a way that $x=u_{1}$, and  $u_{2}\in \mathcal{A}_S\cup \mathcal{B}_S$ whenever $r\ge 2$. We partition $V(C_n)$ into $X=\{u_1,\ldots  ,u_r\}$ and $Y=\{u_{r+1},\ldots ,u_n\}$. Notice that Lemma \ref{lema-1} leads to $|S\cap (Y\times V(H))|\ge 6q$.
 
Now, if $r\in \{1,2\}$, then $|S\cap (X\times V(H))|\ge r+1$, which implies that $|S|\ge r+1+6q=n-\lfloor\frac{n}{7}\rfloor+1$. Analogously, if $r=3$, then $|S\cap (X\times V(H))|\ge r$ and so 
  $|S|\ge r+6q=n-\lfloor\frac{n}{7}\rfloor$. 

Finally, if $r\in \{4,5,6\}$, then by Lemma \ref{lemma-conditions-proj} (i) and (ii) we deduce that $|S\cap (X\times V(H))|\ge r$, which implies that $|S|\ge r+6q=n-\lfloor\frac{n}{7}\rfloor$.
\end{proof}

The following result is a direct consequence of Theorem \ref{DoubleDominSpanningSubgraph} and Lemmas \ref{LemmeGammax2LexicographicPathasxH}  and \ref{LemmaDDS-CiclesxHGmmaH=2}.

\begin{proposition}\label{FormulaGammax2LexicographicPnxH=CnxH} For any integer $n\ge 3$ and any graph $H$ with $\gamma(H)=2$,
$$\gamma_{\times 2}(C_n\circ H)=\gamma_{\times 2}(P_n\circ H)= \left\{ \begin{array}{ll}
             n-\lfloor\frac{n}{7}\rfloor+1 & \text{if } \,  n\equiv 1,2 \pmod 7,\\[5pt]
             n-\lfloor\frac{n}{7}\rfloor & \mbox{otherwise.}
                                  \end{array}\right.$$
\end{proposition}

 \subsection{Cases where $\gamma(H)\geq 3$}

To begin this subsection we need to recall the following well-known result. 

\begin{remark} {\rm \cite{book-total-dom} }\label{TotalDominationCyclesPaths}
For any integer $n\ge 3$,
$$\gamma_t(P_n)=\gamma_t(C_n)=
\left\{ \begin{array}{ll}
\frac{n}{2} & if \, n\equiv 0\pmod 4,\\[5pt]
\frac{n+1}{2} & if \,  n\equiv 1,3\pmod 4,\\[5pt]
\frac{n}{2}+1  & if \,  n\equiv 2\pmod 4.
\end{array}\right. $$
\end{remark}

\begin{lemma}\label{lem-new-new}
Let $P_n=u_1u_2\ldots u_n$ be a path of order $n\geq 6$, where consecutive vertices are adjacent, and let $H$ be a graph. If $\gamma(H)\geq 3$, then there exists a $\gamma_{\times 2}(P_n\circ H)$-set $S$ such that $u_n, u_{n-3}\in \mathcal{C}_S$ and $u_{n-1}, u_{n-2}\in \mathcal{A}_S$. 
\end{lemma}

\begin{proof}
Let $S$ be a $\gamma_{\times 2}(P_n\circ H)$-set  that satisfies Lemma \ref{obs-vertice<=2} such that $|\mathcal{A}_S|$ is maximum. First, we observe that $u_{n-1}\in \mathcal{A}_S$ by Lemma \ref{lemma-conditions-proj}. Now, by applying again Lemma \ref{lemma-conditions-proj}, we have that $|S\cap V(H_{u_n})|+|S\cap V(H_{u_{n-2}})|\geq 2$. Hence, without loss of generality we can assume that $u_{n-2}\in \mathcal{A}_S$ and $u_n\in \mathcal{C}_S$ as $|\mathcal{A}_S|$ is maximum. If $u_{n-3}\in \mathcal{C}_S$, then we are done. On the other hand, if $u_{n-3}\notin \mathcal{C}_S$, then as every vertex of $V(H_{u_{n-3}})$ has two neighbours in $S\cap V(H_{u_{n-2}})$, we can redefine $S$ by replacing the vertices in $S\cap V(H_{u_{n-3}})$ with vertices in  $V(H_{u_{n-4}})\cup V(H_{u_{n-5}})$ and obtain a new $\gamma_{\times 2}(P_n\circ H)$-set $S$ satisfying that $u_{n-3}\in \mathcal{C}_S$, as desired. Therefore, the result follows. 
\end{proof}

\begin{proposition}\label{teo-Pn-gamma>=3}
Let $n\geq 3$ be an integer and let $H$ be a graph. If $\gamma(H)\geq 3$, then 
$$\gamma_{\times 2}(P_n\circ H)=2\gamma_t(P_n)=\left\{ \begin{array}{ll}
n & if \, n\equiv 0\pmod 4,\\[5pt]
n+1 & if \,  n\equiv 1,3\pmod 4,\\[5pt]
n+2  & if \,  n\equiv 2\pmod 4.
\end{array}\right. $$
\end{proposition}

\begin{proof}
Since Proposition \ref{prop-inequalities} leads to $\gamma_{\times 2}(P_n\circ H)\leq 2\gamma_t(P_n)$, we only need to prove that $\gamma_{\times 2}(P_n\circ H)\geq 2\gamma_t(P_n)$. We proceed by induction on $n$. By Propositions \ref{prop-Kn-K1n} and \ref{prop-Snm-Knm} we obtain that $\gamma_{\times 2}(P_n\circ H)=2\gamma_t(P_n)$ for $n=3,4$. By Lemma \ref{lemma-conditions-proj} it is easy to see that $\gamma_{\times 2}(P_5\circ H)=2\gamma_t(P_5)$. This establishes the base case. Now, we assume that $n\geq 6$ and that $\gamma_{\times 2}(P_k\circ H)\geq 2\gamma_t(P_k)$ for $k<n$. Let $S$ be a $\gamma_{\times 2}(P_n\circ H)$-set  that satisfies Lemma \ref{lem-new-new}. Let $D=V(P_n\circ H)\setminus (\cup_{i=0}^3 V(H_{u_{n-i}}))$. Notice that $S\cap D$ is a double dominating set of $(P_n\circ H)-D\cong P_{n-4}\circ H$. Hence,  by applying the induction hypothesis,
$$\gamma_{\times 2}(P_n\circ H)\geq \gamma_{\times 2}(P_{n-4}\circ H)+4\geq 2\gamma_t(P_{n-4})+4\geq 2\gamma_t(P_n),$$
as desired. To conclude the proof we apply Remark  \ref{TotalDominationCyclesPaths}.
\end{proof}

\begin{proposition}\label{teo-Cn-gamma>=3}
Let $n\geq 3$ be an integer and let $H$ be a graph. If 
$\gamma(H)\geq 3$, then
$$\gamma_{\times 2}(C_n\circ H)=n.$$
\end{proposition}

\begin{proof}
From Theorem \ref{teo-bounds-x2-2,t} we know  that $\gamma_{\times 2}(C_n\circ H)\leq n$. We only need to prove that $\gamma_{\times 2}(C_n\circ H)\geq n$. Let $S$ be a $\gamma_{\times 2}(G\circ H)$-set that satisfies Lemma \ref{obs-vertice<=2}. Since $\gamma(H)\ge 3$, by Lemma \ref{lemma-conditions-proj} (iii) we deduce that
 $$2 \gamma_{\times 2}(C_n\circ H)=2|S|=\displaystyle\sum_{x\in V(C_n)}\sum_{u\in N(x)}|S\cap V(H_{u})|\geq 2n.$$
Therefore, the result follows.
\end{proof}

\end{document}